\documentclass[11pt]{amsart}
\usepackage{amsmath}
\usepackage{amsthm}
\usepackage{amsfonts}
\usepackage{amssymb}
\usepackage[all]{xy}
\usepackage{alltt}

\theoremstyle{plain}
\newtheorem{prop}{Proposition}[section]

\newtheorem{thm}[prop]{Theorem}
\newtheorem{cor}[prop]{Corollary}
\theoremstyle{definition}

\newtheorem{remark}[prop]{Remark}

\newcommand{\ints}{\ensuremath{\mathbb{Z}}}
\newcommand{\intmod}[1]{\ensuremath{\mathbb{Z}/{(#1)}}}

\newcommand{\rats}{\ensuremath{\mathbb{Q}}}

\newcommand{\n}{\noindent}

\DeclareMathOperator{\Hom}{Hom}

\newcommand{\smsh}{\ensuremath{{\scriptstyle \wedge}}}             

\newcommand{\wdg}{\ensuremath{{\scriptstyle \vee}}}

\DeclareMathOperator{\im}{im}

\newcommand{\map}[3]{\ensuremath{#1 : #2 \longrightarrow #3}}

\DeclareMathOperator{\ann}{ann}

\newcommand{\thicksub}[1]{\ensuremath{{{\bf thick}\langle#1\rangle}}}

\DeclareMathOperator{\StMod}{StMod}

\newcommand{\loc}[1]{\ensuremath{{{\bf loc}\langle#1\rangle}}}

\newcommand{\dr}{\ensuremath{{\mathcal{D}(R)}}}

\newcommand{\modcat}[1]{\ensuremath{{\bf Mod}(#1)}}

\DeclareMathOperator{\chr}{char}

\DeclareMathOperator{\End1}{End}
\newcommand{\Endo}[1]{\ensuremath{\End1(#1)}}

\newcommand{\idp}{\ensuremath{\mathfrak{p}}}

\begin{document}  

% Article information
\title{Triangulations of Projective Modules}
\date{\today}

% Author information
\author{Mark Hovey}
\address{Department of Mathematics \\ Wesleyan University
\\ Middletown, CT 06459}
\email{hovey@member.ams.org}

\author{Keir Lockridge}
\address{Department of Mathematics \\ Wesleyan University
\\ Middletown, CT 06459}
\email{keir@alumni.rice.edu}

% AMS information

\begin{abstract} We show that the category of projective modules over a graded commutative ring admits a triangulation with respect to module suspension if and only if the ring is a finite product of graded fields and exterior algebras on one generator over a graded field (with a unit in the appropriate degree).  We also classify the ungraded commutative rings for which the category of projective modules admits a triangulation with respect to the identity suspension.  Applications to two analogues of the generating hypothesis in algebraic topology are given, and we translate our results into the setting of modules over a symmetric ring spectrum or $S$-algebra.
\end{abstract}

\maketitle

\tableofcontents

\section{Introduction}

In the stable category of spectra, Freyd's generating hypothesis (\cite[\S 9]{freyd}) is the conjecture that any map of finite spectra inducing the trivial map of homotopy groups must itself be trivial.  The global version of this conjecture -- that this is true for all maps of spectra -- is easily seen to be false.  The present work was motivated by the following question:  what must be true about a triangulated category in order for it to support a global version of the generating hypothesis?  We will show that this question is related to the following strictly algebraic one:  for which rings does the associated category of projective modules admit a triangulation?  This question can be addressed without reference to any form of the generating hypothesis; this is done in section 2 for commutative rings.  In section 3, we make explicit the relationship between this algebraic question and two forms of the generating hypothesis; there is slightly more at issue, however, than whether projective modules admit triangulations.  In the final section, we simply translate our work into the setting of ring spectra (symmetric ring spectra or $S$-algebras), where we define semisimple and von Neumann regular ring spectra and discuss their classification.

First, we establish notation.  Let $R$ be a graded ring.  By a module we mean a graded right $R$-module.  Let $\mathcal{P}$ be the category of projective $R$-modules, and let $\mathcal{P}_f$ denote the category of finitely generated projective $R$-modules.  For any $R$-module $M$, write $M[n]$ for the shifted module with $M[n]_i = M_{i-n}$.  For any element $m \in M$, denote by $|m|$ the degree of $m$.  Write $\Hom_k(M,N)$ for degree $k$ module maps from $M$ to $N$; observe that $\Hom_k(M,N) = \Hom_0(M[k],N)$.  If $M$ is an $R$-$R$-bimodule (or, if $R$ is graded commutative), then, for any element $x \in R$ of degree $i$, let $x \cdot M$ denote the right module map from $M[i]$ to $M$ induced by left multiplication by $x$.

We assume that the reader is somewhat familiar with triangulated categories.  In brief, a triangulated category is an additive category $\mathcal{T}$ together with an automorphism $\Sigma$ of $\mathcal{T}$ called {\em suspension} and a collection of diagrams called {\em exact triangles} of the form
\[ \xymatrix{A \ar[r] & B \ar[r] & C \ar[r] & \Sigma A}\]
satisfying several axioms (see \cite[A.1.1]{axiomatic} or \cite{neeman} or \cite{weibel}).  Using the suspension functor, one can view the morphism sets in $\mathcal{T}$ as graded groups:  define $[X, Y]_k = \Hom_{\mathcal{T}}(\Sigma^k X, Y)$ and
\[ [X,Y]_* = \bigoplus_{k \in \ints} [X,Y]_k.\]
It seems natural, therefore, to consider one of two possibilities for $\Sigma$ on $\mathcal{P}$ or $\mathcal{P}_f$.  If $\Sigma(-) = (-)[1]$, then $[M,N]_*$ is the graded group of graded module maps from $M$ to $N$.  Alternatively, we could consider $\Sigma = {\bf 1}$, the identity functor.  When $R$ is concentrated in degree zero, one then obtains  a triangulation of the category of ungraded projective (or finitely generated projective) $R$-modules by identifying this category with the thick subcategory of $\mathcal{P}$ (or $\mathcal{P}_f$) generated by the modules concentrated in degree zero.  Usually, we will assume that the suspension functor is of the form $\Sigma = (-)[n]$ for some $n$.

For convenience, we will use the term $\Delta^n$-{\em ring} to refer to any ring for which $\mathcal{P}$ admits a triangulation with suspension $\Sigma = (-)[n]$ and $\Delta^n_f$-{\em ring} for any ring for which $\mathcal{P}_f$ admits a triangulation with suspension $\Sigma = (-)[n]$.  Our goal is to characterize the graded commutative rings $R$ for which the categories $\mathcal{P}$ and $\mathcal{P}_f$ admit triangulations.  In particular, we prove

\begin{thm}  Let $\mathcal{P}$ be the category of projective modules over a graded commutative ring $R$.  $\mathcal{P}$ admits a triangulation with suspension $\Sigma = (-)[1]$ if and only if $R \cong R_1 \times \cdots \times R_n$, where each factor ring $R_i$ is either a graded field $k$ or an exterior algebra $k[x]/(x^2)$ over a graded field $k$ containing a unit of degree $3|x|+1$.
\label{mainthm1}
\end{thm}

In the ungraded case, we have

\begin{thm}  Let $\mathcal{P}$ be the category of projective modules over a commutative ring $R$.  $\mathcal{P}$ admits a triangulation with suspension $\Sigma = {\bf 1}$ if and only if $R \cong R_1 \times \cdots \times R_n$, where each factor ring $R_i$ is either a field $k$, an exterior algebra $k[x]/(x^2)$ over a field $k$ of characteristic 2, or $T/(4)$, where $T$ is a complete $2$-ring.
\label{mainthm0}
\end{thm}

\n  A complete 2-ring is a complete local discrete valuation ring of characteristic zero whose maximal ideal is generated by 2 (see \cite[p. 223]{matsumura}).  In \cite{muro}, it is shown that the category of finitely generated projective $T/(4)$-modules admits a unique triangulation, and we use their methods to construct triangulations for the rings appearing in the above two theorems.

Let $\mathcal{T}$ be a triangulated category, and let $S \in \mathcal{T}$ be a distinguished object.  Write $\pi_*(-)$ for the functor $[S,-]_*$.  We say that $\mathcal{T}$ {\em satisfies the global generating hypothesis} if $\pi_*$ is a faithful functor from $\mathcal{T}$ to the category of graded right modules over $\pi_* S$. In the following application of Theorem \ref{mainthm1}, we assume that $\mathcal{T}$ is a monogenic stable homotopy category as defined in \cite{axiomatic}.  Certain conclusions can be drawn with weaker hypotheses; this should be clear in the proof.

\begin{cor} Let $\mathcal{T}$ be a monogenic stable homotopy category with unit object $S$.  $\mathcal{T}$ satisfies the global generating hypothesis if and only if \begin{enumerate}
\item $R \cong R_1 \times \cdots \times R_n$, where $R_i$ is either a graded field $k$ or an exterior algebra $k[x]/(x^2)$ over a graded field containing a unit in degree $3|x|+1$ ($R$ is a $\Delta^1$-ring), and
\item for every factor ring of $R$ of the form $k[x]/(x^2)$, $x \cdot \pi_* C \neq 0$, where $C$ is the cofiber of $x\cdot S$. 
\end{enumerate}
\label{gghcor}
\end{cor}

The {\em thick subcategory generated by $S$} ($\thicksub{S}$) is the smallest full subcategory of $\mathcal{T}$ that contains $S$ and is closed under suspension, retraction, and exact triangles.  We say that $\mathcal{T}$ {\em satisfies the strong generating hypothesis} if, for any map $\map{f}{X}{Y}$ with $X \in \thicksub{S}$ and $Y \in \mathcal{T}$, $\pi_* f = 0$ implies $f = 0$.  Regarding the strong generating hypothesis, we have the following corollary.  If $\idp$ is a prime ideal, write $R_\idp$ for the localization of $R$ at $\idp$.

\begin{cor} Assume the hypotheses of Corollary \ref{gghcor}.  If $\mathcal{T}$ satisfies the strong generating hypothesis, then, for any prime ideal $\idp$, $R_\idp$ is either a graded field $k$ or an exterior algebra $k[x]/(x^2)$ with a unit in degree $3|x| + 1$.  If $R$ is local or Noetherian, then $\mathcal{T}$ satisfies the strong generating hypothesis if and only if it satisfies the global generating hypothesis.
\label{gghsghcor}
\end{cor}

We would like to thank Sunil Chebolu, who informed us of his work with Benson, Christensen, and Min\'{a}\v{c} on the global generating hypothesis in the stable module category.  It is his presentation of the material in \cite{sunil} that led us to consider the questions raised in the present paper.

\section{Triangulations of projective modules}

It is in this section that we prove Theorems \ref{mainthm1} and \ref{mainthm0}.  Before we begin, we make a simple observation that we will use frequently.  Suppose $\mathcal{P}$ (or $\mathcal{P}_f$) is triangulated.  If we apply the functor $[R, -]_0$ to an exact triangle
\[ \xymatrix{A \ar[r]^-f & B \ar[r]^g & C \ar[r]^h & \Sigma A,} \]
then we must obtain a long exact sequence of $R$-modules.  Hence, any exact triangle is exact as a sequence of $R$-modules. (`Exact' at $\Sigma A$ means $\ker (-\Sigma f) = \im h$.)

First we show that if $\mathcal{P}$ admits a triangulation, then $R$ must be a quasi-Frobenius ring, and if $\mathcal{P}_f$ admits a triangulation, then $R$ must be an IF-ring.  A ring $T$ is {\em quasi-Frobenius} if it is right Noetherian and right self-injective.  This condition is right-left symmetric, and the quasi-Frobenius rings are exactly the rings for which the collections of projective and injective modules coincide.  In fact, a ring is quasi-Frobenius if and only if every projective module is injective, if and only if every injective module is projective.  A ring $T$ is {\em injective-flat} (an IF-ring) if every injective module is flat.  Note that IF-rings are coherent (\cite[6.9]{faith}), so all $\Delta^n$-rings and $\Delta^n_f$-rings are coherent.  More information on quasi-Frobenius and IF-rings may be found in \cite[\S 15]{lam} and \cite[\S 6]{faith}.

\begin{prop} If $\mathcal{P}$ admits a triangulation, then $R$ is quasi-Frobenius.
\label{timpqf}
\end{prop}
\begin{proof}  Let $M$ be an $R$-module.  There is a map $\map{f}{A}{B}$ of projective modules whose cokernel is $M$.  This map must lie in an exact triangle in $\mathcal{P}$,
\[ \xymatrix{A \ar[r]^-f & B \ar[r] & C \ar[r] & \Sigma A.} \]
Exactness at $B$ implies that $M$ is isomorphic to a submodule of the projective module $C$.  If $M$ is injective, then it must be a summand of $C$ and therefore projective.  Hence, every injective $R$-module is projective, and $R$ is quasi-Frobenius.
\end{proof}

\begin{prop} If $\mathcal{P}_f$ admits a triangulation, then $R$ is an IF-ring. \label{timpif}\end{prop}
\begin{proof} In light of the proof of Proposition \ref{timpqf}, we see that every finitely presented module embeds in a finitely generated projective module.  By \cite[6.8]{faith}, $R$ is an IF-ring.
\end{proof}

\begin{remark} When $R$ is quasi-Frobenius, one can form $\StMod(R)$, the {\em stable module category} of $R$.  The objects of $\StMod(R)$ are $R$-modules and the morphisms are $R$-module maps modulo an equivalence relation:  two maps are equivalent if their difference factors through a projective module.  $\StMod(R)$ is a triangulated category.  If $M$ is an $R$ module, then the {\em Heller shift} of $M$, written $\Omega M$, is the kernel of a projective cover $\xymatrix@1{P(M) \ar[r] & M}$.  This descends to a well-defined, invertible endomorphism of the stable module category, and $\Omega^{-1}$ is the suspension functor for the triangulation of $\StMod(R)$.  In Proposition \ref{timpqf}, the exact triangle gives rise to three short exact sequences:
\[\xymatrix{0 \ar[r] & M \ar[r] & C \ar[r] & \Sigma \ker f \ar[r] &0\\
	0 \ar[r] &\ker f \ar[r] & A \ar[r] & \im f \ar[r] & 0\\
	0 \ar[r] & \im f \ar[r] & B \ar[r] & M \ar[r] &0.}\]
\n Together they imply that, if $\mathcal{P}$ admits a triangulation, then $\Omega^3 \Sigma M \cong M$.  This observation was made by Heller in \cite{heller}.
\end{remark}

The following proposition shows that a triangulation imposes a severe restriction on the local rings that may occur.

\begin{prop}  Suppose $\mathcal{P}$ (or $\mathcal{P}_f$) admits a triangulation.  If $R$ is a graded commutative local ring with maximal ideal $\mathfrak{m}$, then $\mathfrak{m}$ is principal and contains no nontrivial proper ideals.
\label{maxismin}
\end{prop}
\begin{proof} First, we observe that, even if $R$ is neither graded commutative nor local, it must satisfy the double annihilator condition $\ann_l \ann_r Rx = Rx$.  Let $x \in R$, and let $i = |x|$.  We have an exact triangle in $\mathcal{P}$
\[ \xymatrix{R[i] \ar[r]^-x & R \ar[r]^-\psi & P \ar[r]^-\phi & \Sigma R[i].} \]
Observe that $\im \phi = \ann_r Rx$, and if $y \in \ann_l \ann_r Rx$, then $(y \cdot R) \circ \phi = 0$.  Exactness at $R[i]$ therefore implies that $\ann_l \ann_r Rx= Rx$.

For the remainder of the proof, we assume $R$ is graded commutative and local.  Since $\mathcal{P}$ (or $\mathcal{P}_f$) is triangulated, $R$ is coherent; this implies that $\ann (x)$ is finitely generated for all $x \in R$.

Consider $x \in \mathfrak{m}$.  We now show that $x^2 = 0$.  Since $(x\cdot P) \circ \psi = \psi \circ (x \cdot R) = 0$, we obtain a factorization $x \cdot P = f \circ \phi$, where $\map{f}{\im \phi}{P}$.  Since $\im \phi = \ann (x)$, $(x \cdot P) \circ f = 0$.  Hence, $x^2 \cdot P = 0$.  Since $R$ is local, $P$ is free ($P$ is nontrivial since $x$ is not a unit).  Therefore $x^2 = 0$.

Fix a nonzero element $x \in \mathfrak{m}$.  We next show that $\ann (x) = (x)$.  Note that $P$ is an extension of finitely generated modules and is therefore finitely generated.  Let $e_1, \dots, e_n$ generate $P$, let $\phi_i = \phi(e_i)$, and let $\psi(1) = \Sigma e_i t_i$.  Since $x$ is nonzero, $\im \phi = \ann (x) \subseteq \mathfrak{m}$; hence, $\phi (e_i \phi_i) = \phi_i^2 = 0$, and exactness at $P$ provides an element $z_i \in R$ such that $\psi(z_i) = e_i \phi_i$.  This implies that $t_j z_i = 0$ when $i \neq j$, and $t_i z_i = \phi_i$.  Since $\im \phi = \ann (x)$, the elements $\phi_i$ generate $\ann (x)$.  Let $t = t_1 + \cdots + t_n$ and consider $q = \Sigma \phi_i a_i \in \ann (x)$.  We have $t(\Sigma z_i a_i) = q$, so $\ann (x) \subseteq (t)$.  Since $\psi \circ (x \cdot R) = 0$, we obtain $t_ix = 0$ for all $i$.  Hence, $(t) = \ann (x)$.  Since $x \neq 0$, $t$ is not a unit, so $(t) \subseteq \ann (t) = \ann \ann (x) = (x)$.  This proves that $\ann (x) = (t) = (x)$.  Further, it is worth remarking that $P$ must be free of rank 1.  Suppose $n \geq 2$.  It is clear that $e_1x$ and $e_2 x$ are in the kernel of $\phi$.  Hence, for some $a, b \in R$, $e_1 x = \psi(a)$ and $e_2 x = \psi(b)$.  This means that $e_1 xb - e_2 xa = 0$, forcing $a,b \in \ann (x) = (x)$.  But since $t_i x = 0$ for all $i$, it must be the case that $e_1x = \psi(a) = \Sigma e_i t_i a = 0$, a contradiction. 

Finally, we show that $\mathfrak{m}$ must be the unique, proper, nontrivial, principal ideal of $R$.  Suppose $a\in \mathfrak{m}$ and $b \in R$ have the property that $ab \neq 0$.  Then, $a \in \ann (ab) = (ab)$.  Hence, there is an element $k \in R$ such that $a (1-bk) = 0$.  Since $R$ is local, this forces $a$ to be zero or $b$ to be a unit.  Now consider two nonzero elements $x,y \in \mathfrak{m}$.  We have $x \in \ann (y) = (y)$ and $y \in \ann (x) = (x)$, so $(x) = (y)$.  This forces $\mathfrak{m}$ to be the unique, proper, nontrivial, principal ideal of $R$.  
\end{proof}

\begin{prop} Let $R$ be a graded commutative local ring with residue field $k$.  If $\mathcal{P}$ (or $\mathcal{P}_f$) admits a triangulation with suspension $\Sigma = (-)[n]$, then $R$ is either\begin{enumerate}
\item the graded field $k$,
\item an exterior algebra $k[x]/(x^2)$ with a unit in degree $3|x|+n$, where $\chr k = 2$ if $n = 0$, or
\item $n$ is even and $R \cong T/(4)$, where $T$ is the unique (up to isomorphism) complete $2$-ring with residue field $k$ (of characteristic 2) containing a unit in degree $n$. 
\end{enumerate}
\label{localcase}
\end{prop}
\begin{proof} In the proof of Proposition $\ref{maxismin}$, we showed that the cofiber of $x \cdot R$ is free of rank 1 for any nontrivial element $x$ in the maximal ideal $\mathfrak{m}$.  Hence, any such $x$ fits into an exact triangle of the form
\begin{equation} \xymatrix{R[i] \ar[r]^-x & R \ar[r]^-{vx} & R[j] \ar[r]^-{wx} & R[i+n],}\label{xcof}\end{equation}
where $i = |x|$ and $v$ and $w$ are units.  One can check that $|(vw)^{-1}| = 3i+n$.

Suppose $R$ contains a field.  As such, $R$ is a ring of equal characteristic; since it is complete, it contains a field isomorphic to $k = R/\mathfrak{m}$ (\cite[28.3]{matsumura}).  So either $\mathfrak{m} = 0$ and $R$ is a graded field, or $\mathfrak{m} = (x)$ for some $x \neq 0$ (by Proposition \ref{maxismin}) and $R$ is isomorphic to the exterior algebra $k[x]/(x^2)$.  In the latter case, we just observed that $k$ must contain a unit of degree $3i+n$.  

If $p$ is a unit for all primes $p$, then $R$ contains the field $\rats$.  If $p$ is not a unit for some prime, then either $R$ has characteristic $p$ and contains a field, or $p$ is nonzero and is in the maximal ideal.  If $p$ is in the maximal ideal $\mathfrak{m}$, then by Proposition \ref{maxismin}, we must have $\mathfrak{m} = (p)$.  If $n$ is odd, this is not possible:  since $p$ has degree zero, we saw above that there must be a unit of degree $3\cdot 0 + n$, which is odd.  But if $s$ is an odd unit, then $2s^2 = 0$ by graded commutativity; this forces $2 = 0$, contradicting $\mathfrak{m} = (p)$.

It remains to consider the case where $n$ is even and $\mathfrak{m} = (p)$ for some prime $p$.  Let $T$ be the unique (up to isomorphism) complete $p$-ring with residue field $k$ (see \cite[\S 29]{matsumura}).  By the discussion on p. 225 of \cite{matsumura}, since $R$ is a complete local ring of unequal characteristic, it contains a coefficient ring $A \cong T/(p^2)$.  Further, it is observed in the proof of \cite[29.4]{matsumura} that every element can be expanded as a power series in the generators of the maximal ideal with coefficients in $A$.  Since the maximal ideal is generated by $p \in A$, we obtain $R \cong A$.  To see why $p=2$, consider the rotation of triangle (\ref{xcof}),
\[\xymatrix{R \ar[r]^-{vp} & R[j] \ar[r]^-{wp} & R[n] \ar[r]^-{-p} & R[n].}\]
The map $\kappa$ in the diagram
\[\xymatrix{R \ar[r]^-p \ar@{=}[d] & R \ar[r]^-{vp} \ar[d]^-v & R[j] \ar[r]^-{wp} \ar@{-->}[d]^-\kappa & R[n] \ar@{=}[d]\\
		R \ar[r]^-{vp} & R[j] \ar[r]^-{wp} & R[n] \ar[r]^-{-p} & R[n]} \]
must exist.  Hence, $\kappa v p = w p v = -p\kappa v$, so $2p = 0$.  This forces $p = 2$.

When $n=0$ and $R = k[x]/(x^2)$, an argument similar to the one just given shows that $\chr k = 2$.
\end{proof}

We next show that, for $n = 0$ and $n = 1$, the rings in the conclusion of Proposition \ref{localcase} are $\Delta^n$-rings.

\begin{prop}[\cite{muro}] If $T$ is a complete $2$-ring, then $T/(4)$ is a $\Delta^0$-ring. \label{t4}\end{prop}
\begin{proof} In \cite{muro}, a triangulation of $\mathcal{P}_f$ is constructed for any commutative local ring of characteristic 4 with maximal ideal $(2)$, but the proof in fact shows that $\mathcal{P}$ admits a triangulation.  These rings are exactly the rings of the form $T/(4)$ (where $T$ is a complete $2$-ring).  Certainly, any ring of the form $T/(4)$ is of the type discussed in \cite{muro}, and in Proposition \ref{localcase} it is shown that any commutative local ring of characteristic 4 with maximal ideal $(2)$ is of the form $T/(4)$.
\end{proof}

\begin{prop} Every graded field $k$ is a $\Delta^n$-ring for all $n$.  Every exterior algebra $k[x]/(x^2)$ with a unit in degree $3|x| + n$ is a $\Delta^n$-ring, provided $n$ and $|x|$ are not both even when $\chr k \neq 2$.\label{dga}\end{prop}
\begin{proof} We will use a graded version of the construction presented in \cite{muro}.  Fix $n\geq 0$.  We will work in the category of differential graded modules over a differential graded algebra $A$, where the degree of all derivations is $-n$.  So if $x,y \in A$,
\[ d(xy) = d(x)y + (-1)^{n|x|}xd(y).\]
This category admits a triangulation with suspension functor $(-)[n]$.  For example, if $u \in A$ has degree $i$, then there is an exact triangle
\[\xymatrix{A[i] \ar[r]^-u & A \ar[r] & A \oplus A[i+n] \ar[r] & A[i+n],}\]
where the differential on $A[j]$ is $(-1)^jd$ and the differential $D$ on $A \oplus A[i+n]$ is
\[ D(a,b) = (da + ub, (-1)^{i+n}db).\]

Let $A = k$ with zero differential.  Trivially, the homology of any differential graded $A$-module is projective, so homology induces an equivalence of categories from $D(A)$ (the derived category of $A$) to $\mathcal{P} = \modcat{k}$ by Proposition \ref{ghchar} (4).

For the exterior algebra $k[x]/(x^2)$ with a unit $v$ in degree $3|x| + n$, we use the differential graded algebra constructed in \cite{muro}.  Let $i = |x|$, and let $a$ and $u$ be symbols with degrees $|a| = 2i+n$ and $|u| = i$.  Let $A = k\langle a, u \rangle/I$, where $I$ is the two sided ideal generated by the homogeneous elements $a^2$ and $au + ua + v$ (here we see why the existence of the unit $v$ is necessary).  Define the differential on $A$ by $da = u^2$ and $du = 0$.  One can check that this differential is well-defined; if $i$ and $n$ have opposite parity, then $v$ has odd degree, and so $\chr k = 2$ by graded commutativity and signs do not matter.  If $i$ and $n$ are both odd, then the signs work out independent of the characteristic of $k$.  The differential is not well-defined if $i$ and $n$ are both even and $\chr k \neq 2$; fortunately, if $n = 0$, $\chr k = 2$ is forced (see Proposition \ref{localcase}).

As in \cite{muro}, it is straightforward to check that $H_* A \cong k[x]/(x^2)$, where $x$ is the homology class of $u$, and for any differential graded $A$-module $M$, $H_* M$ is projective.  Again we see that homology induces an equivalence of categories from $D(A)$ to $\mathcal{P}$.
\end{proof}

Observe that, for any integer $n$, every $\Delta^n$-ring is a $\Delta^n_f$-ring.  For if $\mathcal{P}$ admits a triangulation, then $\thicksub{R}$ admits a triangulation.  Since $R$ must be Noetherian, $\thicksub{R} = \mathcal{P}_f$.  Propositions \ref{t4} and \ref{dga} now imply

\begin{prop} Let $n \in \{0,1\}$.  Every ring $R$ in the conclusion of Proposition \ref{localcase} is a $\Delta^n$-ring.  Hence, the classes of graded commutative local $\Delta^n$-rings and graded commutative local $\Delta^n_f$-rings coincide.\label{localchar}\end{prop}

According to \cite[15.27]{lam}, a commutative ring is quasi-Frobenius if and only if it is a finite product of local Artinian rings with simple socle.  The following proposition allows us to restrict our attention to the local case.  Combined with the above characterization of commutative local $\Delta^n$-rings for $n \in \{0,1\}$, this completes the proofs of Theorems \ref{mainthm1} and \ref{mainthm0}.

\begin{prop}  Suppose $R \cong A \times B$.  The category of projective $R$-modules admits a triangulation if and only if the categories of projective $A$-modules and projective $B$-modules each admit a triangulation.  The same is true for finitely generated projective modules.
\label{ringprod}
\end{prop}
\begin{proof}  This is a consequence of the fact that $\modcat{R} \cong \modcat{A} \times \modcat{B}$.
\end{proof}

As a final note, we broaden the scope of the second statement of Proposition \ref{localchar}.

\begin{prop} Let $n \in \{0,1\}$.  The classes of graded commutative Noetherian $\Delta^n$-rings and graded commutative Noetherian $\Delta^n_f$-rings coincide. \label{noethchar}\end{prop}
\begin{proof}  It suffices to check that every commutative Noetherian $\Delta^n_f$-ring $R$ is a $\Delta^n$-ring.  Since $R$ must be an IF-ring (Proposition \ref{timpif}), it must be a commutative Noetherian self-injective ring (\cite[6.9]{faith}).  This makes $R$ quasi-Frobenius by definition, and therefore a product of local rings, each of which must be a $\Delta^n$-ring by Proposition \ref{localchar} (since each local ring is a $\Delta^n_f$-ring).  Hence, $R$ is a $\Delta^n$-ring.
\end{proof}

\section{The generating hypothesis}

We now give an application of the above classification to the global generating hypothesis.  First, we make a general observation.  A triangulated category $\mathcal{T}$ is {\em cocomplete} if arbitrary coproducts exist in $\mathcal{T}$.  We call a $\pi_* S$-module $M$ {\em realizable} if there is an object $X \in \mathcal{T}$ such that $\pi_*X = M$.  If $\mathcal{T}$ is cocomplete, then idempotents split in $\mathcal{T}$ (\cite{bockstedt}), so every projective $\pi_* S$-module is realizable as $\pi_* Y$, where $Y \in \langle S \rangle^0$, the collection of retracts of coproducts of suspensions of $S$.  Hence, $\pi_*$ induces an equivalence of categories
\[\map{\Phi}{\langle S \rangle^0}{\mathcal{P}}\]
(recall that $\mathcal{P}$ is the category of projective $\pi_* S$-modules).  Write $\loc{S}$ for the {\em localizing subcategory generated by $S$} (the smallest full subcategory of $\mathcal{T}$ containing $S$ that is thick and closed under arbitrary coproducts).  If $\mathcal{T} = \loc{S}$, then we say that $S$ {\em generates} $\mathcal{T}$.  We have the following equivalent forms of the global generating hypothesis; for this proposition, we do not require $\pi_* S$ to be graded commutative.

\begin{prop}  Let $\mathcal{T}$ be a cocomplete triangulated category.  The following are equivalent: \begin{enumerate}
\item $\mathcal{T}$ satisfies the global generating hypothesis.
\item $\mathcal{T} = \langle S \rangle^0$, the collection of retracts of arbitrary coproducts of suspensions of $S$.
\item The functor $\pi_*(-)$ is full and faithful.
\end{enumerate}
If $S$ is a compact generator of $\mathcal{T}$, then these conditions are equivalent to \begin{enumerate}
\item[(4)] Every realizable module is projective.
\end{enumerate}
If any one of the above conditions is satisfied, then the category $\mathcal{P}$ of projective right $\pi_* S$-modules admits a triangulation with shift functor $\Sigma = (-)[1]$.
\label{ghchar}
\end{prop}
\begin{proof} 
\n $(1) \implies (2)$. Assume $\mathcal{T}$ satisfies the global generating hypothesis.  Let $Y$ be an object of $\mathcal{T}$.  Using a generating set for the $\pi_*S$-module $\pi_*Y$, construct a map $\map{f}{X}{Y}$ such that $X \in \langle S \rangle^0$ and $\pi_* f$ is surjective.  This map fits into an exact triangle
\[\xymatrix{X \ar[r]^-f & Y \ar[r]^-g & Z \ar[r] & \Sigma X.}\]
Since $\pi_* f$ is surjective, $\pi_* g = 0$.  By the global generating hypothesis, $g = 0$, forcing $Y$ to be a retract of $X$.  Hence, $\mathcal{T} = \langle S \rangle^0$.

\n $(2) \implies (3)$.  This implication is trivial, since $\Phi$ is an equivalence of categories.

\n $(3) \implies (1)$.  This implication is true by definition of the global generating hypothesis.

\n $(2) \implies (4)$.  This implication is trivial.

\n $(4) \implies (2)$.  Assume $S$ is compact.  For any $X \in \loc{S}$, $\pi_* X = 0$ if and only if $X$ is trivial.  Consequently, for any map $f$ in $\loc{S}$, $\pi_* f$ is an isomorphism if and only if $f$ is an equivalence.  Now fix $X \in \loc{S}$.  Since $\pi_* X$ is projective, there is a map $\map{f}{Y}{X}$ such that $Y \in \langle S \rangle^0 \subseteq \loc{S}$ and $\pi_* f$ is an isomorphism.  Hence, $X$ is equivalent to $Y$.  Since $\mathcal{T} = \loc{S}$, the implication is established.

If any of these conditions hold, then $\Phi$ may be used to endow $\mathcal{P}$ with a triangulation.  In this triangulation, $\Sigma = (-)[1]$ since $\pi_* \Sigma X = \pi_{*-1}X = \pi_* X [1]$.
\end{proof}

\begin{remark}  Suppose $\mathcal{T}$ satisfies Brown Representability (i.e., cohomological functors are representable); for example, this holds in any cocomplete, compactly generated triangulated category.  If $I$ is an injective $\pi_* S$-module, then the functor $\Hom_0(\pi_*(-),I)$ is a cohomological functor on $\mathcal{T}$, hence representable by an object $E(I) \in \mathcal{T}$.  In this situation, $\pi_* E(I) = I$, so every injective module is realizable.  If every realizable module is projective, then every injective module is projective.  Similarly, if every realizable module is injective, then every projective module is injective.  In either case, $\pi_* S$ is quasi-Frobenius, and the injective, projective, and realizable modules coincide.  Hence, if $\mathcal{T}$ satisfies Brown Representability, then it satisfies the strong generating hypothesis if and only if the realizable modules and the projective modules coincide, if and only if the realizable modules and the injective modules coincide. \label{brownrem}\end{remark}

In the graded commutative case, then, we see that if $\mathcal{T}$ satisfies the global generating hypothesis, then $\pi_* S$ must be a product of fields and exterior algebras on one generator by Theorem \ref{mainthm1}.  However, the converse of this statement if false; the triangulation on $\mathcal{P}$ may not be induced by $\pi_*$.  For example, consider the stable module category $\StMod(k[G])$ associated to a finite $p$-group $G$ and field $k$ of characteristic $p$.  In $\StMod(k[G])$, $\pi_* S = \hat{H}(G;k)$, the Tate cohomology of $G$.  For $n\geq 1$, $\hat{H}(\intmod{3^n}; \mathbb{F}_3) \cong \mathbb{F}_3[y,y^{-1}][x]/(x^2)$, where $|x| = 1$ and $|y| = 2$.  However, it is shown in \cite{sunil} that $\StMod(\mathbb{F}_3[\intmod{3}])$ satisfies the global generating hypothesis, though $\StMod(\mathbb{F}_3[\intmod{3^n}])$ does not for $n \geq 2$.  The relevant condition is included in Corollary \ref{gghcor}, which we now prove.

\begin{proof}[Proof of Corollary \ref{gghcor}]
First, note that every stable homotopy category has a symmetric monoidal structure (for which $S$ is the unit) compatible with the triangulation.  This forces $R = \pi_* S$ to be graded commutative (\cite[A.2.1]{axiomatic}).  Further, we may use the monoidal product to take any element $x \in \pi_* S$ and obtain a map $x \cdot X$ for any $X \in \mathcal{T}$.  For example, if $e \in \pi_* S$ is idempotent, then $e \cdot X$ is an idempotent endomorphism of $X$.  Since idempotents split, there is a decomposition $X \simeq Y \wdg Z$ such that $\pi_*Y  \cong e\pi_* X$ and $\pi_* Z = (1-e)\pi_* X$ (see \cite[1.4.8]{axiomatic}).

Suppose $\mathcal{T}$ satisfies the global generating hypothesis.  By Proposition \ref{ghchar}, the category $\mathcal{P}$ of projective modules over $\pi_* S$ admits a triangulation with suspension functor $(-)[1]$.  By Theorem \ref{mainthm1}, condition (1) is satisfied, and $\pi_* S \cong R_1 \times \cdots \times R_n$, where $R_i$ is either a graded field $k$ or an exterior algebra $k[x]/(x^2)$.  Since idempotents split in $\mathcal{T}$, $S \cong A_1 \wdg \cdots \wdg A_n$, where $\pi_* A_i = R_i$.  The cofiber $C_i$ of $x\cdot A_i$ is a summand of $C$, the cofiber of $x \cdot S$.  Since $\pi_* C_i$ must be a nontrivial free $R_i$-module, $x \cdot C_i \neq 0$.  This proves (2).

Conversely, if condition (1) holds, then $R$ and $S$ admit decompositions as above.  We now show that it suffices to assume $R$ is local.  Let $\mathcal{T}_i = \loc{A_i}$; $\loc{A_i}$ is an ideal by \cite[1.4.6]{axiomatic}.  Every $X \in \mathcal{T}$ admits a decomposition $X \cong X_1 \wdg \cdots \wdg X_n$, where $X_i \in \mathcal{T}_i$ (take $X_i = A_i \smsh X$).  We claim that this decomposition is orthogonal, in that $\Hom_{\mathcal{T}}(\mathcal{T}_i, \mathcal{T}_j) = 0$ whenever $i \neq j$ (i.e., every map from an object in $\mathcal{T}_i$ to an object in $\mathcal{T}_j$ is trivial).  First, observe that if $i \neq j$, then any map $\map{f}{A_i}{A_j}$ is trivial, as follows.  It suffices to show that the induced map $\map{f_*}{R_i}{R_j}$ is zero.  Let $e_k \in \pi_* S$ be the idempotent corresponding to $R_k$.  We now have $f_* (x) = f_* (e_i x) = e_i f_* (x) = 0$, for all $x \in R_i$.  So indeed, $[A_i, A_j]_* = 0$.  Now, $[A_i, -]_*$ vanishes on $\loc{A_j}$ since $A_i$ is compact; therefore $[-,X]_*$ vanishes on $\loc{A_i}$ for any $X \in \loc{A_j}$.  In summary, there is an orthogonal decomposition
\[\mathcal{T} = \loc{S} \cong \mathcal{T}_1 \wdg \cdots \wdg \mathcal{T}_n. \]
We now see that $\pi_*$ is faithful on $\mathcal{T}$ if and only if the functors $[A_i,-]_*$ are faithful on $\mathcal{T}_i$ for $i = 1, \dots, n$.  It therefore suffices to assume that $R$ is a local ring.

We must now show that if $R = \pi_* S$ is local and satisfies conditions (1) and (2), then $\pi_*$ is faithful.  By Proposition \ref{ghchar} (4), we need only check that $\pi_* X$ is always projective (or free, since $R$ is local).  If $R$ is a graded field, then this is trivially true.  If $R \cong k[x]/(x^2)$, then condition (2) tells us that $\pi_* C$ is free, where $C$ is the cofiber of $x\cdot S$. Since every map of free $k[x]/(x^2)$-modules is the coproduct of a trivial map, an isomorphism, and multiplication by $x$, it is easy to check that $\pi_* X$ is free for all $X \in \thicksub{S}$.  For arbitrary $X\in \mathcal{T}$, $\pi_* X$ is the direct limit of a system of modules of the form $\pi_* X_\alpha$, where $X_\alpha \in \thicksub{S}$ (\cite[2.3.11]{axiomatic}). Since any direct limit of flat modules is flat, $\pi_* X$ is flat.  Over $k[x]/(x^2)$, flat implies free.  This completes the proof.

\end{proof}

\begin{remark}  Let $\mathcal{T}$ be the category of projective modules over $\intmod{4}$ endowed with the triangulation associated to $\Sigma = {\bf 1}$, and let $S = \intmod{4}$.  It is, of course, the case that $\pi_*$ acts faithfully in this situation.  However, the ring $\pi_* S = \bigoplus_{i\in \ints}\intmod{4}$ is not graded commutative, so this does not contradict Corollary \ref{gghcor}.\end{remark}

Finally, we study the strong generating hypothesis.  Write $\langle S \rangle_f^0$ for the collection of retracts of finite coproducts of suspensions of $S$.  As above, $\pi_*$ induces an equivalence of categories
\[\map{\Phi_f}{\langle S \rangle_f^0}{\mathcal{P}_f}.\]
We have the following equivalent forms of the strong generating hypothesis.

\begin{prop}  Let $\mathcal{T}$ be a cocomplete triangulated category, and assume $S$ is compact.  The following are equivalent: \begin{enumerate}
\item $\mathcal{T}$ satisfies the strong generating hypothesis.
\item $\thicksub{S} = \langle S \rangle^f_0$, the collection of retracts of finite coproducts of suspensions of $S$.
\item For all $X \in \thicksub{S}$, $\pi_* X$ is projective.
\end{enumerate}
If any one of the above conditions is satisfied, then the category $\mathcal{P}_f$ of finitely generated projective right $\pi_* S$-modules admits a triangulation with suspension functor $\Sigma = (-)[1]$.
\label{sghchar}
\end{prop}
\begin{proof} 
\n $(1) \iff (2)$.  Suppose $\mathcal{T}$ satisfies the strong generating hypothesis.  As in the proof of Proposition \ref{ghchar}, we see that any $X \in \thicksub{S}$ is a retract of a coproduct of suspensions of $S$.  Since $S$ is compact, so is $X$; it must therefore be a retract of a finite coproduct. Hence, (1) implies (2).  The converse is trivial.

\n $(2) \iff (3)$.  The proof of this equivalence is similar to the proof of its analogue in Proposition \ref{ghchar}.  Note that we do not need to assume $\mathcal{T} = \loc{S}$ since $\thicksub{S} \subseteq \loc{S}$.

If any of these equivalent conditions hold, then $\Phi_f$ may be used to endow $\mathcal{P}_f$ with a triangulation with suspension $\Sigma = (-)[1]$.
\end{proof}

\begin{remark} This remark is a companion to Remark \ref{brownrem}.  Assume $\mathcal{T}$ is a monogenic stable homotopy category and a Brown category (see \cite[4.1.4]{axiomatic}).  For all $X \in \mathcal{T}$, $\pi_* X$ is the direct limit over a system of modules of the form $\pi_* X_\alpha$, where $X_\alpha \in \thicksub{S}$.  Consequently, if $\mathcal{T}$ satisfies the strong generating hypothesis, then every realizable module is flat (\cite[4.4]{lam}), and every flat module is realizable (since $(-) \otimes_{\pi_*S} F$ is representable for any flat module $F$).  On the other hand, representability of cohomological functors implies that arbitrary products of $\pi_* S$ are realizable.  If every realizable module is flat, this forces $\pi_* S$ to be coherent (\cite[4.47]{lam}).  Now, for all $X \in \thicksub{S}$, $\pi_* X$ is a finitely presented flat module and therefore projective (\cite[4.30]{lam}).  In summary, if $\mathcal{T}$ is a monogenic stable homotopy category and Brown category, then it satisfies the strong generating hypothesis if and only if the flat modules and realizable modules coincide.  For example, in the derived category of a ring $R$, every $R$-module is realizable, and the strong generating hypothesis is true if and only if all modules are flat, if and only if $R$ is von Neumann regular (\cite{hlp}).\label{brownrem2}\end{remark}

We next show that, if either form of the generating hypothesis holds, then it also holds locally.  For any $\pi_* S$-module $M$, let $M_\idp$ denote the localization of $M$ at the prime ideal $\idp$.  Suppose $\mathcal{T}$ is a cocomplete triangulated category with compact generator $S$.  The proofs of \cite[2.3.17]{axiomatic} and \cite[3.3.7]{axiomatic} show that there exists a localization functor $L_\idp$ on $\mathcal{T}$ such that $\pi_* L_\idp X \cong (\pi_* X)_\idp$.  We mean `localization functor' in the sense of definition \cite[3.1.1]{axiomatic}, omitting the condition involving smash products, since $\mathcal{T}$ may not have a product structure.  In particular, the natural map $\xymatrix@1{Y \ar[r]^-{\iota_Y} & L_\idp Y}$ induces an isomorphism
\begin{equation} \xymatrix{[L_\idp Y, L_\idp X]_* \ar[r]^-\cong & [Y, L_\idp X]_*.}\label{natiso}\end{equation}
Call an object {\em $\idp$-local} if it lies in the image of $L_\idp$, and let $\mathcal{T}_\idp$ denote the full subcategory of $\idp$-local objects.  Since $S$ is compact, $L_\idp$ commutes with coproducts; hence, $\mathcal{T}_\idp = \loc{L_\idp S}$, and $L_\idp S$ is compact in $\mathcal{T}_\idp$.  

\begin{prop} Suppose $\mathcal{T}$ is a cocomplete triangulated category with compact generator $S$.  If $\mathcal{T}$ satisfies the strong or global generating hypothesis, then so does $\mathcal{T}_\idp$ for all prime ideals $\idp$ in $\pi_* S$ \label{global2local}\end{prop}
\begin{proof}
Consider a map $\map{f}{L_\idp X}{L_\idp Y}$.  Because of the natural isomorphism (\ref{natiso}), $f = 0$ if and only if $\tilde{f} = f\circ \iota_X = 0$.  For any $\alpha \in \pi_* X$, we have a commutative diagram
\[\xymatrix{S \ar[r]^-\alpha \ar[d]_-{\iota_S} & X \ar[d]_-{\iota_X} \ar[rd]^-{\tilde{f}} \\
	L_\idp S \ar[r]_-{L_\idp \alpha} & L_\idp X \ar[r]_-f & L_\idp Y.}\]
Hence, if $[L_\idp S, f]_* = 0$, then $[S, \tilde{f}]_* = 0$.  We immediately see that if $\mathcal{T}$ satisfies the global generating hypothesis, then so does $\mathcal{T}_\idp$.  The same argument applies if $\mathcal{T}$ satisfies the strong generating hypothesis, because $X$ is compact in $\mathcal{T}$ if and only if $L_\idp X$ is compact in $\mathcal{T}_\idp$.
\end{proof}

We end this section with the proof of Corollary \ref{gghsghcor}.

\begin{proof}[Proof of Corollary \ref{gghsghcor}]
Assume $\mathcal{T}$ satisfies the strong generating hypothesis.  Since $\mathcal{T}$ is a monogenic stable homotopy category, the unit object $S$ is a compact generator.  Further, $R = \pi_* S$ is commutative.  By Propositions \ref{global2local} and \ref{sghchar}, $R_\idp$ is a $\Delta^1_f$-ring.  Since it is a graded commutative local ring, it must be $\Delta^1$-ring by Proposition \ref{localchar}.  By Proposition \ref{localcase}, $R$ is either a graded field $k$ or an exterior algebra $k[x]/(x^2)$ over a field with a unit in degree $3|x|+1$.  We turn now to the last statement of the corollary.

Certainly, the global generating hypothesis always implies the strong generating hypothesis.  Assume the strong generating hypothesis is true.  By Proposition \ref{sghchar}, $R$ is a $\Delta^1_f$-ring.  If $R$ is local or Noetherian, then it is also a $\Delta^1$-ring by Proposition \ref{localchar} or \ref{noethchar}.  We now wish to invoke Corollary \ref{gghcor}.  To do so, we must verify condition (2).  Arguing as in the proof of Corollary \ref{gghcor}, it suffices to assume that $R \cong k[x]/(x^2)$.  For this ring, condition (2) is implied by the fact that $\pi_* C$ must be free by the strong generating hypothesis.
\end{proof}

\section{Ring spectra}

In this brief section, we translate our results into the setting of ring spectra, where by `ring spectrum' we mean either a symmetric ring spectrum or an $S$-algebra.  Let $E$ be a ring spectrum, and let $\mathcal{D}(E)$ denote the derived category of $E$-modules.  As in \S 3, write $\langle E \rangle^0$ for the collection of retracts of coproducts of suspensions of $E$; call these objects {\em projective $E$-module spectra}.  Define $E$ to be {\em semisimple} if every $E$-module spectrum is projective (i.e., if $\mathcal{D}(E) = \langle E \rangle^0$), and call $E$ {\em von Neumann regular} if every compact $E$-module spectrum is projective (i.e., if $\thicksub{E} = \langle E \rangle^0_f$).  The relationship between such ring spectra and the generating hypothesis is given by

\begin{prop} Let $E$ be a ring spectrum.  $E$ is semisimple if and only if $\mathcal{D}(E)$ satisfies the global generating hypothesis.  $E$ is von Neumann regular if and only if $\mathcal{D}(E)$ satisfies the strong generating hypothesis. \label{ringspectra}\end{prop}
\begin{proof} The category $\mathcal{D}(E)$ is a cocomplete triangulated category with compact generator $E$.  The characterization of semisimplicity follows from Proposition \ref{ghchar}; the characterization of von Neumann regularity follows from Proposition \ref{sghchar}.
\end{proof}

The following proposition shows that these definitions are consistent with the ones in ordinary ring theory.

\begin{prop}  Let $HR$ denote the Eilenberg-Mac Lane spectrum associated to the ring $R$.  $HR$ is semisimple if and only if $R$ is semisimple, and $HR$ is von Neumann regular if and only if $R$ is von Neumann regular. \end{prop}
\begin{proof}
It is well-known (see, for example, \cite[IV.2.4]{ekmm}) that the derived category of $HR$-module spectra is equivalent to the derived category of $R$-modules, $\dr$.  In \cite[\S 4]{keir} it is shown that $\dr$ satisfies the global generating hypothesis if and only if $R$ is semisimple, and in \cite[1.3]{hlp} it is shown that $\dr$ satisfies the strong generating hypothesis if and only if $R$ is von Neumann regular.
\end{proof}

We would like to have a full classification of the semisimple and von Neumann regular ring spectra; we now summarize the application of the foregoing work to this problem.

\begin{prop}  Let $E$ be a ring spectrum.  If $\pi_* E$ is commutative, then \begin{enumerate}
\item If $E$ is semisimple, then $\pi_*E \cong R_1 \times \cdots \times R_n$, where $R_i$ is either a graded field $k$ or an exterior algebra $k[x]/(x^2)$ over a graded field containing a unit in degree $3|x|+1$ (i.e., $\pi_* E$ is a graded commutative $\Delta^1$-ring).
\item If $E$ is von Neumann regular, then $(\pi_* E)_\idp$ is a graded commutative local $\Delta^1$-ring for every prime ideal $\idp$ of $\pi_* E$.
\end{enumerate}
If $E$ is commutative, then\begin{enumerate}
\item[(3)] $E$ is semisimple if and only if $\pi_* E$ is a graded commutative $\Delta^1$-ring and for every factor ring of $\pi_* E$ of the form $k[x]/(x^2)$, $x \cdot \pi_* C \neq 0$, where $C$ is the cofiber of $x\cdot E$.
\item[(4)] If $\pi_*E$ is local or Noetherian, then $E$ is semisimple if and only if $E$ is von Neumann regular.
\end{enumerate}
\label{spectra}
\end{prop}

According to \cite{ss}, every simplicial, cofibrantly generated, proper, stable model category with a compact generator $P$ is Quillen equivalent to the module category of a certain endomorphism ring spectrum $\Endo{P}$.  This is true in particular for the derived category of a differential graded algebra.  In light of the proof of Proposition \ref{dga}, we see that every graded commutative $\Delta^1$-ring arises as $\pi_* E$ for some (not necessarily commutative) ring spectrum $E$.  Because of Proposition \ref{spectra} (3), however, $\pi_* E$ being a graded commutative $\Delta^1$-ring is not sufficient to conclude that $E$ is semisimple.

\vspace{.5in}

%\bibliographystyle{plain}
%\bibliography{triproj}

\end{document}